\documentclass[11pt,a4paper,reqno]{amsart} 

\usepackage [english]{babel}  

\DeclareFontFamily{U}{mathb}{\hyphenchar\font45}
\DeclareFontShape{U}{mathb}{m}{n}{
      <5> <6> <7> <8> <9> <10>
      <10.95> <12> <14.4> <17.28> <20.74> <24.88>
      mathb10
      }{}
\DeclareSymbolFont{mathb}{U}{mathb}{m}{n}

\DeclareMathSymbol{\sqbullet}{1}{mathb}{"0D}

\allowdisplaybreaks 

\usepackage{amsmath,mathtools}
\usepackage{amssymb}
\usepackage{euscript}
\usepackage{bm}
\usepackage{graphicx}
\usepackage[all]{xy}
\usepackage[dvipsnames]{xcolor}
\usepackage[colorlinks=true,linktocpage=true,pagebackref=false, citecolor=black,linkcolor=black]{hyperref}
\usepackage{pifont}
\usepackage{tikz,pgfplots}
\pgfplotsset{compat=1.10}
\usepgfplotslibrary{fillbetween}
\usetikzlibrary{cd,arrows,decorations.pathmorphing,backgrounds,automata,positioning,fit,matrix}
\usepackage{caption,subcaption}
\usepackage{comment}
\usepackage{xcolor}
\usepackage{enumitem} 

\usetikzlibrary{matrix}
\usetikzlibrary{decorations.pathreplacing, calc, fit, backgrounds}

\pgfplotsset{soldot/.style={color=black,only marks,mark=*}} \pgfplotsset{holdot/.style={color=black,fill=white,only marks,mark=*}}

\newtheorem{thm}{Theorem}[section]

\newtheorem{lem}[thm]{Lemma}

\newtheorem{quest2}[thm]{Question}

\theoremstyle{definition}
\newtheorem{defn}[thm]{Definition}

\newtheorem{remark}[thm]{Remark}

\newtheorem{example}[thm]{Example}

\theoremstyle{remark}

\numberwithin{equation}{section}
\numberwithin{figure}{section}

 \newcommand{\R}{{\mathbb R}}
 \newcommand{\C}{{\mathbb C}}

\newcommand{\Rr}{{\EuScript R}}

\newcommand{\cl}{\operatorname{Cl}}

\newcommand{\x}{{\tt x}} \newcommand{\y}{{\tt y}} 
\newcommand{\z}{{\tt z}} \renewcommand{\t}{{\tt t}}
 
\newcommand{\w}{{\tt w}}

\usepackage{scalerel}

\begin{document}

$\;$

\vspace{-2em}

\title[$K$-holomorphic functions with definable real part]{$K$-holomorphic functions with definable real part}


\date{04/05/2026}

\author{Antonio Carbone}
\address{Antonio Carbone, Dipartimento di Scienze dell'Ambiente e della Prevenzione, Palazzo Turchi di Bagno, C.so Ercole I D'Este, 32, Università di Ferrara, 44121 Ferrara (ITALY)}
\email{antonio.carbone@unife.it}
\thanks{The first author is supported by GNSAGA of INdAM}

\author{Enrico Savi}
\address{Enrico Savi, LAREMA, Université d'Angers, 2 Bd de Lavoisier, 49000 Angers (FRANCE).}
\email{enrico.savi@univ-angers.fr}
\thanks{The second author is supported by GNSAGA of INdAM and by the ANR NewMIRAGE (ANR-23-CE40-0002)}

\begin{abstract}
Let $R$ be a real closed field and $K:=R(i)$ its algebraic closure.  Let $U\subset K^n$ be an open and definable set in a fixed o-minimal structure. In this note, we study the relationship between definability of a $K$-holomorphic function $f=f_1+if_2:U\to K$ and the definability and (strong) $R$-analyticity of its real part $f_1:U\to R$. Our results turn out to be the best possible {in general}, and their precision depends on the considered o-minimal structure. We obtain a complete characterisation in the semialgebraic case.
\end{abstract}

\maketitle 

\section{Introduction}

Let $R$ be a real closed field and $K:=R(i)$, where $i:=\sqrt{-1}$, its algebraic closure. The concept of convergence of power series is not suitable for a general real closed field $R$, hence nor for $K$. Therefore, one cannot define ‘holomorphic functions’ in $K^n$ as those functions admitting local expansions as convergent power series. For a function $f:U\to \C$, defined on a (non-empty) open subset $U$ of $\C^n$, being complex analytic is equivalent to be holomorphic, i.e. continuous and separately complex differentiable with respect to each variable. Thus, Peterzil and Starchenko \cite{PS01,PS03} have used this equivalence to define \textit{$K$-holomorphic} functions $f:U\to K$, where $U$ is an open subset of $K$ (see \S\ref{Ranalytic} for the precise definition). As the field $R$ has a natural order topology, that we will call \textit{Euclidean topology}, one can endow $K$ with the product topology after identifying $K$ with $R^2$ in the usual way. Since $R$ may not be complete nor Archimedean, $K$, equipped with the product topology, is, in general, neither locally compact nor locally connected. That is why Peterzil and Starchenko \cite{PS01,PS03} considered functions that are definable in some o-minimal structure on the real closed field $R$ in order to develop their theory of $K$-holomorphic functions on $K^n$. Under this assumption, they developed a satisfactory theory of $K$-holomorphic functions, that shares many of the properties of the classical holomorphic functions with additional tameness properties induced by definability. In particular, these tameness properties have recently shown to be extremely useful in Hodge theory, we refer the interest reader to \cite{BKT20,BBT23,BKU24} for additional information.

Recall that an \textit{o-minimal structure} on the real closed field $R$ is a collection $\mathfrak{\Rr}:=\{\Rr_n\}_{n\in\mathbb{N}^*}$ of families of subsets of $R^n$ satisfying the following properties:
\begin{itemize}
\item $\Rr_n$ is a Boolean algebra,
\item $\Rr_n$ contains the algebraic subsets of $R^n$,
\item if $X\in \Rr_n$ and $Y\in\Rr_m$, then $X\times Y\in\Rr_{n+m}$,
\item if $\pi:R^n\times R\to R^n$ is the projection onto the first factor and $X\in\Rr_{n+1}$, then $\pi(X)\in\Rr_n$,
\item $\Rr_1$ consists exactly of all the finite unions of points and intervals (of any type). 
\end{itemize}
The elements of $\Rr_n$ are called \textit{definable subsets of} $R^n$. A map $f:X\to Y$, between a definable subset $X\subset R^n$ and a definable subset $Y\subset R^m$, is a \textit{definable map} if its graph $\Gamma_f$ is a definable subset of $R^{n+m}$. Recall that a set $X\subset R^n$ is \textit{semialgebraic} if it is a Boolean combination of sets defined by polynomial equalities and inequalities with coefficients in $R$. As a consequence of the Tarski-Seidenberg theorem \cite[Thm.2.2.1]{BCR98}, the family of the semialgebraic sets constitute an o-minimal structure, which is the `smallest' o-minimal structure, in the sense that it is contained in any other o-minimal structure. In particular, semialgebraic sets and maps are definable in any o-minimal structure on the real closed field $R$. When $R=\R$, the collection of \textit{globally subanalytic sets}, introduced by van den Dries \cite{vdD86}, constitutes a necessary restriction for the class of subanalytic sets to generate an o-minimal structure, which coincides with the o-minimal structure $\R_{\text{an}}$ generated by restricted analytic functions. 
We also refer the reader to \cite{vdD98,vdDM96,YC04} for further {information and developments} on the theory of o-minimal structures, together with their applications to geometry and analysis. For the rest of this article, even if not explicitly mentioned, {whenever} we refer to definable sets and maps we mean definable in a fixed o-minimal structure on a fixed real closed field $R$.

{As opposed to the complex case, a characterisation for the notion of real analytic functions relying on differentiability conditions is not available and, as already mentioned before, the definition in terms of convergent power series is not suitable for general real closed fields.} Thus, one cannot directly extend the notion of {(real)} analytic functions to a {any} real closed field $R$. Recently, Kaiser \cite{Kai16b} proposed such a generalisation to arbitrary real closed fields by introducing the notion of \textit{$R$-analytic functions} and \emph{strongly $R$-analytic functions}, as those definable functions admitting a local or global definable complexification, respectively. We refer to \S\ref{Ranalytic} for the precise definitions. 

The main topic of this paper is the investigation on how definability properties of a $K$-holomorphic function $f:=f_1+if_2:U\to K$ are related to $R$-analyticity of its real part $f_1:U\to R$. As a motivation, in classical complex analysis, if $f:=f_1+if_2$ is a holomorphic function defined on an open subset of $\C^n$, then its real part $f_1$ is a real analytic function. Let $U$ be an open subset of $K^n$ and $f:=f_1+if_2:U\to K$ a definable $K$-holomorphic function. {Then,} it is natural to ask {the following question:} 

\begin{quest2}\label{quest1}
Is it true that if $f$ is definable and $K$-holomorphic, then its real part $f_1$ is definable and strongly $R$-analytic?
\end{quest2}

As a first result, we answer the previous question in the positive, by showing the following:

\begin{thm}\label{prop1}
Let $U\subset K^n$ be an open subset and $f:=f_1+if_2:U\to K$ a $K$-holomorphic function. If $f$ is definable, then $f_1$ is definable and strongly $R$-analytic.
\end{thm}

In the recent work \cite{Car25}, the first author showed that the converse of Theorem \ref{prop1} holds true for Nash (i.e. semialgebraic and real analytic) functions when $R=\R$ \cite[Thm.1.3]{Car25}. It is then natural to ask whether the converse of Theorem \ref{prop1} holds true also for definable functions on a general real closed field $R$, namely, to ask the following:

\begin{quest2}\label{quest2}
Is it true that if $f$ is $K$-holomorphic and $f_1$ is definable and strongly $R$-analytic, then $f$ is also definable?
\end{quest2}

In general, Question \ref{quest2} has a negative answer, as  shown by the following example.

\begin{example}\label{ex}
Let $\log z$ denote the branch of {the} logarithm on $\C\setminus\{z\in \R : z\leq 0\}$ such that $\log1=0$. Consider the open semialgebraic set $U:=\{z=x+iy\in\C : 0<x<y\}$ and the holomorphic function 
$$
f:U\to \C, \quad z:=x+iy\mapsto -i\log z=\arctan\Big(\frac{y}{x}\Big)-\frac{i}{2}\log(x^2+y^2).
$$
The real part $f_1(x,y)=\arctan\big(\tfrac{y}{x}\big)$ is globally subanalytic on $U$ and, by \cite[Thm.A]{Kai16a}, admits a global complexification (i.e. it is strongly $\R$-analytic). While, the imaginary part $f_2(x,y)=-\tfrac{1}{2}\log(x^2+y^2)$ is not globally subanalytic on $U$. In fact, for instance, the restriction of $f_2$ to the half line {$\{z=x+iy\in\C : 2y=x, x>0\}$} is not globally subanalytic. Thus, the holomorphic function $f$ is not globally subanalytic on $U$. By contrast, observe that $f$ is definable in $\R_{\text{an},\exp}$, the o-minimal structure obtained by extending $\R_{\text{an}}$ with the exponential funciton $\exp:\R\to \R$, see \cite{vdDMM94}. \hfill$\sqbullet$
\end{example}

Nevertheless, the converse to Theorem \ref{prop1} holds true locally. A function $f:U\to R$, defined on an open subset $U$ of $R^n$, is \textit{locally definable} if for each point $x\in U$ there exists a definable open neighbourhood $V_x$ of $x$ in $U$ such that the restriction of $f$ to $V_x$ is a definable function. Our next result is the following: 

\begin{thm}\label{main}
Let $U$ be an open subset of $K^n$ and $f:=f_1+if_2:U\to K$ a $K$-holomorphic function. If the real part $f_1$ is locally definable and $R$-analytic, then $f$ is locally definable. 
\end{thm}

It is worthwhile mentioning a couple of remarks in the case $R=\R$:

\begin{remark}
(i) Assume $f:=f_1+if_2:U\to \C$ is holomorphic on the open set $U\subset\C^n$ and assume $f_1$ is definable and $\R$-analytic. Since $f_1$ is definable, the Cauchy-Riemann equations ensure that the partial derivatives of $f_2$ are definable. Hence, \cite{Spe99} ensures that $f_2$ is definable in the Pfaffian closure of the o-minimal structure where $f_1$ is defined. Thus, $f$ is a locally definable holomorphic function which is definable in the Pfaffian closure, therefore, its `(global) tame behaviour' is also guaranteed. Moreover, by Theorem \ref{prop1}, we deduce that $f_1$ is also strongly $\R$-analytic with respect to the Pfaffian closure.

(ii) Kaiser \cite[Thm.A]{Kai16a} showed that globally subanalytic functions admit global complexifications (i.e. they are strongly $\R$-analytic in the sense of Definition \ref{strong} {below}). Thus, Theorem \ref{main} applies to the class of globally subanalytic functions. However, in light of Example \ref{ex}, the global definability of $f:U\to \C$ is not guaranteed in $\R_{\text{an}}$ but it is in the Pfaffian closure of $\R_{\text{an}}$ by (i). \hfill$\sqbullet$
\end{remark}

Thanks to a recent result of Kaiser \cite[Thm.B]{Kai16b} and the techniques developed in this article, we show that for semialgebraic functions Theorem \ref{main} holds true globally. In particular, we answer to Question \ref{quest2} in the positive in this case, by showing the following:

\begin{thm}\label{Nash}
Let $U\subset K^n$ be an open subset and $f:=f_1+if_2:U\to K$ a $K$-holomorphic function. If the real part $f_1$ is a semialgebraic function, then $f$ is also semialgebraic. 
\end{thm}

Let $U$ be a subset of $R^n$. A function $f:U\to R$ is called a \textit{Nash function} if it is semialgebraic and infinitely often differentiable. When $R=\R$, then Nash functions are precisely those functions which are semialgebraic and real analytic \cite[Ch.8]{BCR98}. If $f:U\to K$ is a $K$-holomorphic function, then it is infinitely often differentiable \cite[Thm.2.12]{PS03}. In particular, Theorem \ref{Nash} ensures that if the real part $f_1$ is a Nash function, then the imaginary part $f_2$ is Nash as well. 

We conclude this section by noting the following:

\begin{remark}
As $f=f_1+if_2$ is a $K$-holomorphic function if and only if $-if=f_2-if_1$ is a $K$-holomorphic function, then all our results hold true (as expected) also if we replace the real part $f_1$ by the imaginary part $f_2$. \hfill$\sqbullet$
\end{remark}

\vspace{1em}

\section{Preliminaries}

\subsection{$K$-holomorphic and (strongly) $R$-analytic functions}\label{Ranalytic}

In this section we recall the definitions of $K$-holomorphic functions and (strongly) $R$-analytic functions. We refer the reader to \cite{PS01,PS03} and \cite{Kai16b} for further information and details. First, we recall the definition of $K$-holomorphic functions in one variable \cite[Def.2.26]{PS01}. 

\begin{defn}[$K$-holomorphic function - one variable]
Let $U$ be an open subset of $K$ and $z_0\in U$. A function $f:U\to K$ is \textit{$K$-holomorphic at $z_0$} if
$$
\lim_{h\to 0}\frac{f(z_0+h)-f(z_0)}{h}
$$
exists in $K$ (where the limit is {considered} with respect to the the Euclidean topology of $K\equiv R^2$). The function $f$ is \textit{$K$-holomorphic on} $U$ if it is $K$-holomorphic at every point of $U$. \hfill$\sqbullet$ 
\end{defn}

In analogy with the complex case, Peterzil and Starchenko \cite[Def.2.8]{PS03} have defined $K$-holomorphic functions of several variables as those continuous functions which are separately $K$-holomorphic in each variable. 

\begin{defn}[$K$-holomorphic function - several variables]
Let $U$ be an open subset of $K^n$. A function $f:U\to K$ is \textit{$K$-holomorphic on} $U$ if it is continuous on $U$ and for each $(z_1,\ldots,z_n)\in U$ and $j=1,\ldots,n$ the function 
$$
\{z\in K : (z_1,\ldots,z_{j-1},z,z_{j+1},\ldots,z_n)\in U\}\to K, \quad z\mapsto f(z_1,\ldots,z_{j-1},z,z_{j+1},\ldots,z_n)
$$
is $K$-holomorphic at $z_j$. \hfill$\sqbullet$ 
\end{defn}

In what follows we will need the following lemma:

\begin{lem}\label{principioidentita}
Let $U$ be a definably connected subset of $K^n$ and $f:U\to K$  a definable $K$-holomorphic function. If there exists an open subset $V$ of $R^n\times\{0\}$ such that $f(z)=0$ for each $z\in V$, then $f(z)=0$ for each $z\in U$. 
\end{lem}
\begin{proof}
Let $z\in V$. Using the Cauchy-Riemann equations \cite[Fact2.10]{PS03} we find that all the partial derivatives of $f$ are zero at $z$. As $U$ is definably connected, by \cite[Thm.2.13(2)]{PS03}, we conclude that $f$ is identically zero on $U$, as required.
\end{proof}
 
Next, we recall the definition of $R$-analytic functions \cite[Def.2.1]{Kai16b} and strongly $R$-analytic functions \cite[Def.2.15]{Kai16b}.

\begin{defn}[$R$-analytic function]
Let $U$ be an open subset of $R^n$ and $f:U\to R$ a definable function. Let $x_0\in U$. The function $f$ is \textit{$R$-analytic at $x_0$} {if it admits a local definable} complexification. That is, if there exists an open neighbourhood $U_{x_0}$ of $x_0$ in $U$, an open subset $V_{x_0}$ of $K^n$ such that $U_{x_0}\subset V_{x_0}$ and a definable $K$-holomorphic function $F:V_{x_0}\to K$ such that $F|_{U_{x_0}}=f|_{U_{x_0}}$. The function $f$ is \textit{$R$-analytic on $U$} if it is $R$-analytic at every point of $U$. \hfill$\sqbullet$ 
\end{defn}

Strongly $R$-analytic functions are the global counterpart of $R$-analytic functions.

\begin{defn}[Strongly $R$-analytic function]\label{strong}
Let $U$ be an open subset of $R^n$ and $f:U\to R$ a definable function. The function $f$ is \textit{strongly $R$-analytic} if it admits a definable global complexification. That is, if there exists an open subset $V$ of $K^n$ such that $U\subset V$ and a definable $K$-holomorphic function $F:V\to K$ such that $F|_U=f$. \hfill$\sqbullet$ 
\end{defn}

\subsection{Cartan's calculation}
In this section, we recall an elementary, but brilliant, calculation of Cartan \cite[\S IV.3.5]{Car63} {which allows one to reconstruct a holomorphic function from its real part} in a definable way, that is, without integration.

Let $U\subset K^n$ be an open subset such that $0\in U$ and $f:=f_1+if_2:U\to K$ any function such that $f(0)=0$. Let $z:=(z_1,\ldots,z_n)$ and $w:=(w_1,\ldots,w_n)$. Consider the open set $V\subset K^{2n}$ defined as
\begin{equation}\label{aperto}
V:=\{(z,w)\in K^{2n} : z+iw\in U \text{\ and\ } \overline{z}+i\overline{w}\in U\}.
\end{equation}
Observe that $\big(\tfrac{z}{2},\tfrac{z}{2i}\big)\in V$ if and only if $z\in U$. Moreover, $V\cap R^{2n}=\{(x,y)\in R^{2n} : x+iy\in U\}$. Consider the function 
\begin{equation}\label{g}
g:V\to K, \quad (z,w)\mapsto \frac{1}{2}\Big(f(z+iw)+\overline{f(\overline{z}+i\overline{w})}\Big).
\end{equation}
As $f(0)=0$, by the definition of $g$, we deduce that 
\begin{equation}\label{uguaglianza}
f(z)=2g\Big(\frac{z}{2},\frac{z}{2i}\Big)
\end{equation}
for each $z\in U$. Moreover,
\begin{equation}\label{partereale}
f_1(x+iy)=\frac{1}{2}\Big(f(x+iy)+\overline{f(x+iy)}\Big)=g(x,y)
\end{equation}
for each $(x,y)\in V\cap R^{2n}$.

{We end this section by deducing additional regularity of $g$ induced by regularity of $f$:}

\begin{lem}\label{gholo}
It holds:
\begin{enumerate}[label=\emph{(\roman*)}, ref=(\roman*)]
\item\label{gholo:1} If the function $f$ is definable, then the function $g$ is definable.
\item\label{gholo:2}  If the function $f$ is $K$-holomorphic on $U$, then the function $g$ is $K$-holomorphic on $V$.
\end{enumerate}
\end{lem}

\begin{proof}
Property \ref{gholo:1} follows directly by the definability {of the set $V$ introduced in \eqref{aperto}} and of the complex conjugation. Next, we show property \ref{gholo:2}. Clearly, $g$ is $R$-differentiable (in the sense of \cite[Ch.7]{vdD98}) as a composition of $R$-differentiable functions. Thus, by \cite[Fact.2.27]{PS01}, we only need to check that it satisfies the Cauchy-Riemann equations on $V$. Let $j\in\{1,\dots,n\}$. As the function $(z,w)\mapsto f(z+iw)$ is $K$-holomorphic, we deduce that
\begin{align*}
\frac{\partial g(z,w)}{\partial \bar{z}_j} &=\frac{1}{2}\Big(\frac{\partial f(z+iw)}{\partial \bar{z}_j} +\frac{\partial \overline{f(\overline{z}+i\overline{w})}}{\partial \bar{z}_j}\Big)=\frac{1}{2}\frac{\partial \overline{f(\overline{z}+i\overline{w})}}{\partial \bar{z}_j}\\
&= \frac{1}{2}\frac{\partial \overline{f(\overline{z}_1+i \overline{w}_1,\ldots,z_j+i \overline{w}_j,\ldots,\overline{z}_n+i\overline{w}_n)}}{\partial z_j}\\
&=\frac{1}{2}\overline{\Big(\frac{\partial f(\overline{z}_1+i \overline{w}_1,\ldots,z_j+i \overline{w}_j,\ldots,\overline{z}_n+i\overline{w}_n)}{\partial \overline{z}_j}\Big)}=0
\end{align*}
for each $(z,w):=(z_1,\ldots,z_n,w_1,\ldots w_n)\in V$. A similar computation shows that $g$ satisfies the Cauchy-Riemann equations with respect to the variables $w_1,\ldots,w_n$. We conclude that $g$ is $K$-holomorphic on $V$, as required.
\end{proof}

\vspace{1em}

\section{Proofs of the results}

In this section we show Theorems \ref{prop1}, \ref{main} and \ref{Nash}. First, we show Theorem \ref{prop1}.

\begin{proof}[Proof of Theorem \ref{prop1}]
We may assume that $0\in U$ and $f(0)=0$. Let $V$ be the open set introduced in \eqref{aperto} and $g:V\to K$ the function introduced in \eqref{g}. As $f$ is definable and $K$-holomorphic on $U$, then, by Lemma \ref{gholo}, $g$ is definable and $K$-holomorphic on $V$. Thus, $f_1$ is definable, because $f_1(x+iy)=g(x,y)$ for each $(x,y)\in V\cap R^2=\{(x,y)\in R^2 : x+iy\in U\}$. Moreover, $f_1$ is strongly $R$-analytic, as $g$ is a definable $K$-holomorphic extension of $f_1$ to the open definable set $V\subset K^{2n}$, as required.
\end{proof}

Then, we show Theorem \ref{main}. 

\begin{proof}[Proof of Theorem \ref{main}]
{We are going to show the local definability of $f$ at $x\in U$. Again, by a translation argument, we may suppose that $x=0\in U$ and $f(0)=0$. Up to substitute $U$ with an open definable neighbourhood of $0$ in $U$ where $f_1$ is definable, we may assume in what follows that \textit{$f_1$ is definable.}  Thus, we are only left to show that}: \textit{There exists an open neighbourhood $W$ of $0$ in $K^n$ such that the restriction $f|_{W}$ is definable.} 

Let $V$ be the open set introduced in \eqref{aperto} and $g:V\to K$ the function introduced in \eqref{g}. By Lemma \ref{gholo}\ref{gholo:2}, the function $g$ is $K$-holomorphic on $V$. As $f_1$ is definable and $R$-analytic, then there exist
\begin{itemize}
\item an open neighbourhood $U_0$ of $0$ in $U$,
\item an open neighbourhood $V_0$ of $0$ in $K^{2n}$ such that $U_0\subset V_0$,
\item a definable $K$-holomorphic function $F:V_0\to K$ such that $F|_{U_0}=f_1|_{U_0}$.
\end{itemize}
Observe that $V_0\cap V$ is an open {and definable} neighbourhood of $0$ in $K^{2n}$. {By \cite[Prop.2.18]{vdD98} $V_0\cap V$ is uniquely decomposed in a finite number of definably connected components and each of them is both open and closed in $V_0\cap V$. Thus, let $V'_0\subset K^{2n}$ be the definably connected component of $V_0\cap V$ containing $0$, which is still open in $K^{2n}$.} 
By \eqref{partereale}, we have $g(x,y)=f_1(x+iy)=F(x,y)$ for each $(x,y)\in V'_0\cap R^2$. As $g$ is $K$-holomorphic on $V'_0$ and $V'_0$ definably connected, by Lemma \ref{principioidentita}, we have that $g(z,w)=F(z,w)$ for each $(z,w)\in V'_0$. Thus, the restriction $g|_{V'_0}$ is definable, as $F$ is definable on $V_0$, so its restriction to the definable set $V'_0$ is still definable. The set $W:=\{z\in K^n : \big(\tfrac{z}{2},\tfrac{z}{2i}\big)\in V'_0\}$ is an open definable neighbourhood of $0$ in $U$. By \eqref{uguaglianza}, we conclude that $f|_W$ is definable, as required.
\end{proof}

We are left to show Theorem \ref{Nash}, but first, we need some preparation. We start with the following lemma, which is the analogous of \cite[Prop.2.4]{Car25} in this more general context. Although the proof is essentially the same as the one given in \cite[Prop.2.4]{Car25}, we include it here, with all the details on how to extend the arguments, for the reader’s convenience.

\begin{lem}\label{polinomio}
Let $U\subset K^n$ be an open subset and $f:U\to K$ a semialgebraic $K$-holomorphic function. Then, there exists a non-zero polynomial $P\in K[\z,\t]:=K[\z_1,\ldots,\z_n,\t]$ such that $P(z,f(z))=0$ for each $z\in U$.
\end{lem}
\begin{proof}
Let $x:=(x_1,\ldots,x_n)$, $y:=(y_1,\ldots,y_n)$ and $z:=(z_1,\ldots,z_n)$. We identify $K^n$ with $R^{2n}$ by setting $z:=x+iy$ and we regard $R^n$ as the subset $R^n\times\{0\}$. We may write $f=f_1+if_2$, where $f_1,f_2:U\to R$. As $U$ is semialgebraic, it has finitely many semialgebraically connected components $U_1,\ldots,U_s$ \cite[Thm.2.4.4]{BCR98}. Assume that for each $j=1,\ldots,s$ there exists a non-zero polynomial $P_j\in K[\z,\t]$ such that $P_j(z,f(z))=0$ for each $z\in U_j$. Then, the non-zero polynomial $P:=\prod_{j=1}^s P_j$ satisfies $P(z,f(z))=0$ for each $z\in U$. In particular, we may assume: \textit{$U$ is semialgebraically connected}.

As $f$ is semialgebraic, then $f_1$ and $f_2$ are semialgebraic functions. Thus, there exist two non-zero polynomials $P_1,P_2\in R[\x,\y,\t]:=R[\x_1,\ldots,\x_n,\y_1,\ldots,\y_n,\t]$ such that
$$
P_1(x,y,f_1(x+iy))=P_2(x,y,f_2(x+iy))=0
$$
for each $x+iy\in U$. By the identity principle for polynomials, the set 
$$
\{y_0\in R^n :  \, P_1(\x,y_0,\t)=  0 \, \,  \text{or} \, \, P_2(\x,y_0,\t)= 0\}
$$ 
is nowhere dense in $R^n$ with respect to the Euclidean topology. Up to a translation, we may assume that $U\cap R^n$ is an open neighbourhood of 0 in $R^n$ and that the polynomials 
$$
Q_1(\z,\t):=P_1(\z,0,\t) \quad   \text{and} \quad  Q_2(\z,\t):=P_2(\z,0,-i\t)
$$
are not the zero polynomials in $K[\z,\t]$. By the fact that 
$
P_1(x,0,f_1(x))=P_2(x,0,f_2(x))=0
$
for each $x\in U\cap R^n$, we deduce that 
$$
Q_1(x,f_1(x))=P_1(x,0,f_1(x))=0  \quad \text{and} \quad Q_2(x,i f_2(x))=P_2(x,0,f_2(x))=0
$$
for each $x\in U\cap R^n$. 

Let $R(\z,\t)\in K[\z,\t]$ be the resultant in $K[\z,\t,\w]$ with respect to the variable $\w$ of $Q_1(\z,\w)$ and $Q_2(\z,\t-\w)$. As $Q_1(\z,\w)$ and $Q_2(\z,\t-\w)$ have no common irreducible factors in $K[\z,\t,\w]$ of positive degree with respect to the variable $\w$, then by \cite[\S3.6, Prop.1(ii)]{CLO15} and its proof, $R(\z,\t)$ is not the zero polynomial and there exist two polynomials $A,B\in K[\z,\t,\w]$ such that
$$
R(\z,\t)=A(\z,\t,\w)Q_1(\z,\w)+B(\z,\t,\w)Q_2(\z,\t-\w).
$$
{By evaluating $R(\z,\t)\in K[\z,\t,\w]$ at $(x,f(x),f_1(x))$, for $x\in U\cap R^n$, we get:
\begin{equation}\label{Qbullet}
R(x,f(x))=A(x,f(x),f_1(x))Q_1(x,f_1(x))+B(x,f(x),f_1(x))Q_2(x,if_2(x))=0,
\end{equation}
as $Q_1(x,f_1(x))=Q_2(x,+if_2(x))=0$ for every $x\in U\cap R^n$.} The map $U\to K, \ z\mapsto R(z,f(z))$ is $K$-holomorphic on $U$, because it is a composition of $K$-holomorphic {functions}. As $U$ is semialgebraically connected, by \eqref{Qbullet} and Lemma \ref{principioidentita}, we conclude that $R(z,f(z))=0$ for each $z\in U$, as required. 
\end{proof}

Next, we show the following:

\begin{lem}\label{connes}
Let $U\subset R^n$ be an open semialgebraic subset which is, in addition, semialgebraically connected and $\Delta\subset U$ any semialgebraic subset. If the dimension of $\Delta$ is $\leq n-2$, then $U\setminus \Delta$ is semialgebraically connected.
\end{lem}
\begin{proof}
As $U\subset K^n$ is semialgebraically connected, then for every $x,y\in U\setminus \Delta$ there exist a continuous and semialgebraic path $\alpha:[0,1]\to U$ such that $\alpha(0)=x$ and $\alpha(1)=y$ \cite[Prop.2.5.13]{BCR98}. Let us construct a continuous and semialgebraic path $\beta:[0,1]\to U\setminus \Delta$ such that $\beta(0)=x$ and $\beta(1)=y$. 

Let $\partial U:=\cl(U)\setminus U$, where $\cl(U)$ denotes the closure of $U$ in $R^n$ with respect to the Euclidean topology. As $U$ is open, $\partial U=\varnothing$ if and only if $U=R^n$. In what follows, we assume $\partial U\neq \varnothing$ as in the case $U=R^n$ the proof, suitably simplified, works as well. Observe that $\partial U$ is a closed semialgebraic set, as $U$ is an open semialgebaic subset of $R^n$. Given $x:=(x_1,\ldots,x_n)\in R^n$, we denote by $\|x\|:=(x_1^2+\ldots+x_n^2)^{1/2}$ the Euclidean norm of $x$. As $[0,1]\subset R$ is a closed and bounded semialgebraic set, $\alpha:[0,1]\to U$ is continuous and semialgebraic and $\partial U$ is closed and non-empty, the number
$$
M:=\frac{1}{2}\min_{t\in[0,1]}\{\inf\{||\alpha(t)-y|| : y\in\partial U\}\}
$$ 
is well defined and strictly positive by \cite[Prop.2.2.8]{BCR98} and \cite[Thm.2.5.8]{BCR98}. Define
$$
V:=\{x\in U : \inf\{\|x-\alpha(t)\| : t\in [0,1]\}< M\}
$$
which is, using again \cite[Prop.2.2.8]{BCR98}, an open semialgebraic neighbourhood of $\alpha([0,1])$ in $U$. By the definition of $M$, we have that $\cl(V)\subset U$. Moreover, $\cl(V)$ is a closed and bounded semialgebraic set, as $\alpha([0,1])$ is so by \cite[Thm.2.5.8]{BCR98}. Then, by \cite[Thm.9.2.1]{BCR98}, there exists a finite semialgebraic triangulation of $\cl(V)$ relative to the family of semialgebraic sets $\{\cl(V)\cap \Delta,\{x\},\{y\}\}$, that is, there is a finite simplicial complex $S:=(\sigma_i)_{i=1,\dots,t}$, a subcomplex $S'$ of $S$ and a continuous semialgebraic homeomorphism $\Phi:|S|\to V$ such that $\Phi(|S'|)=(\cl(V)\cap \Delta)\cup\{x,y\}$, where $|S|:=\bigcup_{\sigma\in S} \sigma$ and $|S'|:=\bigcup_{\sigma\in S'} \sigma$ denote the realisations of $S$ and $S'$, respectively. Observe that, as the semialgebraic dimension of $(\cl(V)\cap \Delta)\cup\{x,y\}$ is $\leq 2n-2$, then every simplex $\sigma\in S'$ has dimension strictly smaller than $2n-1$.
	
Denote by $S_d$ the set of all the simplices of $S$ of dimension $d\in\{0,\dots,n\}$ and by $b(\sigma)$ the barycenter of a simplex $\sigma\in S$. Let $\sigma, \sigma'\in S_0$ be the simplices such that $\Phi(\sigma)=x$ and $\Phi(\sigma')=y$. As $V$ is open and $x,y\in V$, then there exist two simplices $\tau,\tau'\in S_n$ such that $\sigma\subset \tau$ and $\sigma'\subset \tau'$. As $V$ is open and $\Phi^{-1}\circ \alpha$ is a continuous semialgebraic path between $\sigma$ and $\sigma'$, then there is a sequence of simplices $\tau_1,\dots,\tau_s\in S_n$ such that $\tau_1:=\tau$, $\tau_s:=\tau'$, $\tau_i\cap\tau_{i+1}\in S_{n-1}$ and $\Phi^{-1}(\alpha([0,1])\subset \tau_1\cup \ldots\cup\tau_s$. {Consider the piecewise affine map $\gamma_i:[0,1]\to |S|$ such that $\gamma_i(0)=b(\tau_i)$, $\gamma_i(1/2)=b(\tau_i\cap \tau_{i+1})$ and $\gamma_i(1)=b(\tau_{i+1})$ for every $i\in\{1,\dots,s-1\}$.} Observe that $\gamma_i([0,1])\cap \sigma''=\varnothing$ for every $\sigma''\in S_d$ with $d<2n-1$ and $i\in\{1,\dots,s-1\}$. Consider $\gamma_0:[0,1]\to \tau$ and $\gamma_{s}:[0,1]\to \tau'$ the parametrised segments such that $\gamma_0(0)=\sigma$, $\gamma_0(1)=b(\tau)$, $\gamma_{s}(0)=b(\tau')$ and $\gamma_{s}(1)=\sigma'$. Then, the semialgebraic path $\gamma:[0,1]\to |S|$ defined as the product 
\[
\gamma:=\gamma_0\cdot\gamma_1\cdot\ldots \cdot\gamma_{s-1}\cdot\gamma_{s}
\]
satisfies that $\gamma(0)=\sigma$, $\gamma(1)=\sigma'$ and $\gamma([0,1])\cap\tau=\varnothing$ for every $\tau\in S'\setminus\{\sigma,\sigma'\}$. Then, the semialgebraic path $\beta:[0,1]\to U$ defined as $\beta:=\Phi\circ\gamma$ is continuous and satisfies $\beta(0)=x$, $\beta(1)=y$ and $\beta([0,1])\cap \Delta=\varnothing$, as $\beta([0,1])\subset V$ and $V\cap \Delta\subset\Phi(S'\setminus\{\sigma,\sigma'\})$. That is, $\beta:[0,1]\to U\setminus\Delta$ is a semialgebraic and continuous path such that $\beta(0)=x$ and $\beta(1)=y$. Then, by \cite[Prop.2.5.13]{BCR98}, we conclude that $U\setminus \Delta$ is semialgebraically connected, as required. 
\end{proof}

We are ready to show Theorem \ref{Nash}.

\begin{proof}[Proof of Theorem \ref{Nash}]
As being semialgebraic is a property that is local on the semialgebraically connected components, we may assume that $U$ is semialgebraically connected. We may assume, in addition, that $0\in U$. By \cite[Thm.B]{Kai16a}, $f_1$ is $R$-analytic. Thus, by Theorem \ref{main}, there exists an open ball $B\subset U$ centred in 0 such that $f$ is semialgebraic on $B$. By Lemma \ref{polinomio}, there exists a non-zero polynomial $P\in K[\z,\t]$ such that $P(z,f(z))=0$ for each $z\in B$. As the function $z\mapsto P(z,f(z))$ is $K$-holomorphic on $U$ and vanishes identically on $B$, by \cite[Thm.2.13(2)]{PS03}, then $P(z,f(z))=0$ for each $z\in U$, {as} $U$ is semialgebraically connected. We are going to prove that the graph $\Gamma_f\subset K^{n+1}$ of $f$ is a semialgebraic set.

Let $D(\z)\in K[\z]$ be the discriminant of the polynomial $P(\z,\t)$ with respect to the variable $\t$ and denote by $\Delta:=\{z\in K^n:D(z)=0\}$ the discriminant locus of $P(\z,\t)$. As $\Delta$ is an algebraic subset of $K^n$ defined by a single polynomial equation, then $\Delta$ has semialgebraic dimension $2n-2$ and the set $U':=U\setminus\Delta$ is an open and dense subset of $U$. Consider the projection $\pi:K^{n+1}\to K^n, (z,t)\mapsto z$. As $\Delta\subset K^n$ is the discriminant locus of $P(\z,\t)$, we have that $\pi^{-1}(U')\cap X$ is a semialgebraic subset of $X$ and $\pi|_{\pi^{-1}(U')\cap X}: \pi^{-1}(U')\cap X\to U'$ is a Nash covering of finite degree {by the implicit function theorem \cite[Cor.2.9.8]{BCR98}}. Then, the graph $\Gamma|_{f|_{U'}}$ of $f|_{U'}$ is a semialgebraically connected component of $\pi^{-1}(U')\cap X$, as $f$ is continuous and $U'$ is semialgebraically connected by Lemma \ref{connes}. In particular, $\Gamma|_{f|_{U'}}$ is semialgebraic. Denote by $\cl_{U\times K}(\Gamma|_{f|_{U'}})$ the Euclidean closure of $\Gamma|_{f|_{U'}}$ in $U\times K$. As $f$ is continuous, $\Gamma_f$ is closed in $U\times K$, hence $\cl_{U\times K}(\Gamma|_{f|_{U'}})\subset\Gamma_f$. On the other hand, as $U'$ is dense in $U$ and $f$ is continuous, every $(z,t)\in \Gamma_f\setminus\Gamma_{f|_{U'}}$ is an accumulation point of $\Gamma_{f|_{U'}}$. This shows that $\Gamma_f=\cl_{U\times K}(\Gamma|_{f|_{U'}})$, so $\Gamma_f$ is semialgebraic, as the closure of a semialgebraic set is. We conclude that $f$ is a semialgebraic function, as required. 
\end{proof}

\bibliographystyle{amsalpha}

\end{document}